\documentclass[10pt]{amsart}
\usepackage{amssymb,amsmath}
\usepackage{enumerate}
\usepackage{graphicx}
\usepackage{color}
\usepackage{mathrsfs}
\usepackage{float}

\numberwithin{equation}{section}

\newtheorem{theorem}{Theorem}[section]
\newtheorem{proposition}[theorem]{Proposition}
\newtheorem{lemma}[theorem]{Lemma}
\newtheorem{corollary}[theorem]{Corollary}
\theoremstyle{definition}

\newtheorem{remark}[theorem]{Remark}

\newcommand{\sint}{\int}

\DeclareMathOperator{\BUC}{BUC}

\begin{document}
\title[Distributionally robust random walks]{Limits of random walks with distributionally robust transition probabilities}
 \author{Daniel Bartl
 \and Stephan Eckstein
 \and Michael Kupper}
 \address{University of Vienna}
\email{daniel.bartl@univie.ac.at}
\address{University of Konstanz}
\email{stephan.eckstein@uni-konstanz.de}
\address{University of Konstanz}
\email{kupper@uni-konstanz.de}
\date{\today}

\begin{abstract}
We consider a nonlinear random walk which, in each time step, is free to choose its own transition probability within a neighborhood (w.r.t.\ Wasserstein distance) of the transition probability of a fixed L\'evy process. In analogy to the classical framework we show that, when passing from discrete to continuous time via a scaling limit, this nonlinear random walk gives rise to a nonlinear semigroup. We explicitly compute the generator of this semigroup and corresponding PDE as a perturbation of the generator of the initial L\'evy process.
\end{abstract}

\maketitle
\setcounter{equation}{0}


\section{Introduction and main results}
L\'evy processes are mathematically tractable and therefore often used to model certain real-world phenomena.
This bears the task of correctly specifying / estimating the corresponding parameters, e.g., drift and variance in case of a Brownian motion. In many situations this can only be achieved up to a certain degree of \emph{uncertainty}.
For this reason, Peng \cite{peng2007g} introduced his nonlinear Brownian motion and started a systematic investigation of this object.
The nonlinear Brownian motion is defined via a nonlinear PDE and, heuristically speaking, within each infinitesimal time increment it is allowed to select its parameters (drift and variance) within a given fixed set.
Accordingly, a nonlinear Feynman-Kac formula makes it possible to compute the worst case expectations of certain functions of the random process.
Several works followed this \emph{parametric} nonlinearization approach to L\'evy processes, see, e.g., Hu and Peng \cite{hu2009gl}, Neufeld and Nutz \cite{neufeld2017nonlinear}, Denk et al.~\cite{denk2020semigroup} and K\"uhn \cite{kuhn2018viscosity}.

On the other hand, in discrete time where no mathematical limitations force one to restrict to parametric uncertainty, a more natural and general nonlinearization of a given (baseline) random walk is of \emph{nonparametric} nature.
We start with a random walk which is the discrete-time restriction of an $\mathbb{R}^d$-valued L\'evy process starting in zero, whose marginal laws we denote by $(\mu_t)_{t\geq 0}$.
For instance, $\mu_t$ can be the normal distribution with mean 0 and variance $t$ in which case we end up with a Gaussian random walk.

For a fixed parameter $\delta\geq0$ representing the level of freedom (or uncertainty) and $n\in\mathbb{N}$, the  \emph{nonlinear random walk} with time index $\mathbb{T}=\{0=t_0<t_1<t_2<\cdots\}\subset\mathbb{R}_+$ is defined as follows:
for each time step $t_n\rightsquigarrow t_{n+1}$,
the nonlinear random walk is allowed to select its transition probability within
the neighborhood of size  $\delta \Delta t_{n+1}$ of the transition probability
$\mu_{\Delta t_{n+1}}$
of our baseline random walk,
where $\Delta t_{n+1}:=t_{n+1}-t_n$ and the neighborhood is taken w.r.t.\ the $p$-th \emph{Wasserstein distance}
$\mathcal{W}_p$.\footnote{
	For $\mu,\nu\in\mathcal{P}_p(\mathbb{R}^d)$ (the set of Borel probabilities on $\mathbb{R}^d$ with finite $p$-th moment), define
	$\mathcal{W}_p(\mu,\nu):=\inf\big\{  \int_{\mathbb{R}^d\times\mathbb{R}^d} |y-x|^p\,\gamma(dx,dy) : \gamma\in\mathrm{Cpl}(\mu,\nu) \big\}^{1/p} $
	where $\mathrm{Cpl}(\mu,\nu)$ is the set of all Borel probabilities on $\mathbb{R}^d\times\mathbb{R}^d$ with first and second marginal $\mu$ and $\nu$, respectively.
	Throughout, $|\cdot|$ is the Euclidean norm.}
This means that, conditioned on the event that the nonlinear random walk takes
the value $x\in\mathbb{R}^d$ at time $t_n$, the worst possible expected value
of an arbitrary function $f\in C_0(\mathbb{R}^d)$ at time $t_{n+1}$ is given by
\[ S(\Delta t_{n+1}) f(x)
:=\sup\Big\{ \sint_{\mathbb{R}^d} f(x+y)\,\nu(dy) : \nu \text{ such that } \mathcal{W}_p(\mu_{\Delta t_{n+1}}, \nu)\leq \delta \Delta t_{n+1}\Big\}. \]
Recall here that $C_0(\mathbb{R}^d)$ is the set of continuous function vanishing at infinity.
Iterating this scheme, conditioned on the event that the nonlinear random walk starts in $x$ at time 0, the worst possible expectation at time $t_n\in\mathbb{T}$ is given by
\begin{align}
\label{eq:def.cal.S.intro}
\mathscr{S}^{\mathbb{T}}(t_n)f(x)
:=S(t_1-t_0) \circ \dots \circ S(t_{n}-t_{n-1})f(x).
\end{align}
In conclusion, the corresponding processes follow the same heuristics as the nonlinear Brownian motion and can be seen as a discrete time nonparametric reincarnation thereof.

Regarding the computation of $\mathscr{S}^{\mathbb{T}}$ we
stumble on a recurring scheme in discrete time: while definitions are
mathematically simple, explicit computations are often very challenging.
Here this is evident as $S$ and therefore $\mathscr{S}^{\mathbb{T}}$ are results of (iterated, nonparametric, and infinite dimensional) control problems.
In the following, we shall show that when passing from small to infinitesimal time steps, the $\mathscr{S}^{\mathbb{T}}$'s give rise to a nonlinear semigroup and that a computation of the limit is possible via a nonlinear PDE.

\vspace{0.5em}
For the rest of this article we shall fix $p\in(1,\infty)$ and assume that our initial L\'evy process has finite $p$-th moment, i.e., $\int_{\mathbb{R}^d} |x|^p\,\mu_1(dx)<\infty$.
For convenience, for every $n\in\mathbb{N}$ consider dyadic numbers $\mathbb{T}_n:=2^{-n}\mathbb{N}_0$ and set $\mathscr{S}^n(t):=\mathscr{S}^{\mathbb{T}_n}(t_n)\circ S(t-t_n)$ for $t\geq 0$, where $t_n\in \mathbb{T}_n$ is the closest dyadic number prior to $t$.

\begin{proposition}[Semigroup]
	\label{prop:semigroup.intro}
	Both $\mathscr{S}^n$ and $\mathscr{S}:=\lim_{n\to\infty} \mathscr{S}^n$ are well-defined and the family $(\mathscr{S}(t))_{t\geq 0}$ defines a sublinear semigroup on $C_0(\mathbb{R}^d)$.
	More precisely, for every $s,t\geq 0$ and $x\in\mathbb{R}^d$, 
	\begin{enumerate}[(i)]
		\item $\mathscr{S}(t)$ maps $C_0(\mathbb{R}^d)$ to itself and $\mathscr{S}(t)\circ\mathscr{S}(s)=\mathscr{S}(t+s)$,
		\item $\mathscr{S}(t)(\cdot)(x)\colon C_0(\mathbb{R}^d)\to\mathbb{R}$ is sublinear, increasing, and maps zero to zero, and $\mathscr{S}(t)$ is contractive, i.e., $\|\mathscr{S}(t) f - \mathscr{S}(t) g\|_\infty \leq \|f - g \|_\infty$ for all $f, g \in C_0(\mathbb{R}^d)$.
	\end{enumerate}
\end{proposition}

Now that the semigroup property is established, we can state our main result.

\begin{theorem}[Feynman-Kac]
	\label{thm:pde.intro}
	Let $f\in C_0(\mathbb{R}^d)$ and define $u\colon[0,\infty)\times\mathbb{R}^d\to\mathbb{R}$  via $u(t,x):=\mathscr{S}(t)f(x)$.
	Then $u$ is a viscosity solution of
	\[ \begin{cases}
	\partial_t u(t,x) = A^\mu u(t,\cdot)(x) + \delta|\nabla u(t,x)| &\text{for } (t,x)\in(0,\infty)\times\mathbb{R}^d,\\
	u(0,x)=f(x)&\text{for }x\in\mathbb{R}^d,
	\end{cases}\]
	where $A^\mu$ is the generator of the initial L\'evy process.
\end{theorem}

Here $\nabla$ denotes the spatial derivative.
Moreover, the notion of viscosity solution we consider here is that of \cite{denk2020semigroup}, and we refer to the discussion before and after Theorem \ref{thm:pde.convex} below for the definition and comments on uniqueness.
\begin{remark}
	\label{intro:rem}
	Starting with an arbitrary L\'evy process, Theorem \ref{thm:pde.intro} ensures that the limiting
	semigroup $\mathscr{S}$ corresponds to a nonlinear L\'evy process with parametric drift uncertainty.
	The interesting feature of Theorem \ref{thm:pde.intro} is that, even though we consider \emph{nonparametric uncertainty} in discrete time 
	leading to robustifications that are structurally unconstrained, in the limit we end up with a process bearing only \emph{parametric uncertainty}, which is a drift-perturbed version of the initial L\'evy process. For instance, if we start with a Brownian motion  with generator $A^\mu=\tfrac{1}{2}\Delta$, we end up with a $g$-Brownian motion with generator $\tfrac{1}{2}\Delta+\delta|\nabla|$. More generally, if the initial L\'evy process is a Brownian motion with drift $\gamma\in\mathbb{R}^d$ and covariance matrix $\Sigma\in \mathbb{R}^{d\times d}$, then a quick computation shows that the PDE of Theorem \ref{thm:pde.intro} takes the form
	\[ \partial_t u(t,x)
	=\frac{1}{2} \sum_{i,j=1}^d \Sigma_{ij} \partial_{ij} u(t,x) +  \max_{\eta\in \Gamma }  \sum_{i=1}^d \eta_i \partial_i u(t,x), \]
	where $\Gamma:=\{\eta\in\mathbb{R}^d : |\eta-\gamma|\leq\delta\}$.
	In case of a $g$-Brownian motion, the solution of the PDE can be 
	represented in terms of a $g$-expectation, see \cite{peng1997}.
	Likewise, based on the theory of backward stochastic differential equations 
	with jumps, $g$-expectations also exist for certain jump filtrations (e.g., 
	the one generated by a Brownian motion and a Poisson random measure), see  
	\cite{kazi2016quadratic, liu2019jump, royer2006backward}.
	This leads to corresponding nonlinear processes with parametric drift 
	uncertainty. 
	Finally, we note that similar drift perturbations also arise for related scaling limits, see, e.g., \cite[Proposition 11]{pichler2020quantification}.
\end{remark}

In the following chapter we consider a convex generalization of the above 
setting: in the definition of $S(\Delta t_{n+1})$, instead of considering all $\nu$ in the 
$\delta \Delta t_{n+1}$ neighborhood of $\mu_{\Delta t_{n+1}}$, we take into account all $\nu$ but 
penalized by their distance to $\mu_{\Delta t_{n+1}}$.
In the limit this gives a convex semigroup for which the generator includes a 
convex perturbation in $\nabla u$ (instead of the absolute value), see Theorem 
\ref{thm:pde.convex}.

Finally, let us point out that numerical computation of nonlinear PDEs like the 
ones resulting from Theorem \ref{thm:pde.intro} and Remark \ref{intro:rem} has 
received a lot of attention in recent years and by now efficient methods are 
available, see, e.g., \cite{beck2019machine, raissi2018deep} and references 
therein.

\vspace{1em}
\noindent
{\bf Possible extensions and related literature.}
There are several natural variations of the results in this paper.
For instance, one can ask which effect additional constraints on the measures $\nu$ appearing in the definition of $S(t)$ might have.
Concretely, what would happen if one allows only for those $\nu$ which (additional to being in a Wasserstein neighborhood of $\mu_t$) have the same mean as $\mu_t$, or if one replaces the Wasserstein distance by its martingale version \cite{beiglbock2016problem}.
In the latter case, when changing the scaling of the radius from $\delta t$ to $\delta t^2$, one could guess the PDE to be
\[ \begin{cases}
\partial_t u(t,x) = \Delta  u(t,\cdot)(x) + \delta|\nabla^2 u(t,x)| &\text{for } (t,x)\in(0,\infty)\times\mathbb{R}^d,\\
u(0,x)=f(x)&\text{for }x\in\mathbb{R}^d,
\end{cases}\]
in case that the underlying L{\'e}vy process is the Brownian motion.
However, with the exact methods of this paper, this can be made rigorous only with a (unnatural) technical twist in definition of $S(t)$ and understanding the full picture would require considerations beyond the scope of this paper.

In a similar spirit, it would be interesting to start with transition probabilities $(\mu^n_t)_{t}$ which approximate $(\mu_t)_t$ (e.g.\ a Binomial random walk which converges to a Brownian motion).
A (parametric) variant of this was done by Dolinsky, Nutz, and Soner \cite{dolinsky2012weak} for Binomial random walks with freedom in the Bernoulli-parameter.
Related, one could ask whether Donsker-type results hold, i.e., whether the family of laws of the nonlinear random walks (on the path space) has a limiting family.

Finally, let us highlight the connection to distributionally robust optimization (DRO) using the Wasserstein distance.
In DRO, the basic task consists of computing $\inf_{\lambda} S(t)f^{\lambda}$, where $(f^\lambda)_\lambda$ is a parametrized family of function; we refer to \cite{bartl2020computational, blanchet2019quantifying, esfahani2018data, feng2018quantifying} for recent results and applications.
Here duality arguments often help to compute the (infinite dimensional) optimization problem appearing in the definition of $S(t)$.
In multi-step versions of DRO (e.g., time-consistent utility maximization with 
Markovian endowment \cite{bartl2019exponential}), the computation of 
$\mathscr{S}^n(t)f$ is the key element.
Related multi-step versions also occur in the literature on robust Markov chains with interval probabilities (see \cite{eckstein2019extended, vskulj2009discrete} and references therein), and in particular on robust Markov decision processes \cite{wiesemann2013robust, yang2017convex}.
As $\mathscr{S}(t)f$ can be seen as a proxy for $\mathscr{S}^n(t)f$ for large $n$, a natural question is whether the results in the current paper can be used as an approximation tool for these multi-step versions of DRO.
This also motivates studying the speed of convergence $\mathscr{S}^n(t)f \rightarrow \mathscr{S}(t)f$.

\section{Convex version and proof of main results}

Let $\varphi\colon [0,+\infty)\to[0,+\infty]$ be a lower semicontinuous, convex, and increasing function which is not constant and such that $\varphi(0)=0$.
Assume that $x\mapsto \varphi(x^{1/p})$ is convex, denote by $\varphi^\ast(y):=\sup_{x\geq 0} (xy-\varphi(x))$ for $y\geq 0$ its convex conjugate, and set $\varphi(+\infty):=\varphi^\ast(+\infty)=+\infty$.
For every $f\in C_0(\mathbb{R}^d)$ and $t\geq 0$, we define
\begin{equation}
\label{def:S}
S(t) f(x)
=\sup_{\nu\in\mathcal{P}_p(\mathbb{R}^d)} \Big( \int_{\mathbb{R}^d} f(x+y)\,\nu(dy) - \varphi_t(\mathcal{W}_p(\mu_t,\nu))\Big),
\end{equation}
where $\varphi_t(\cdot):=t\varphi(\cdot/t)$ for $t>0$, $\varphi_0(a) = +\infty$ for $a > 0$, and $\varphi_0(0) = 0$.
The results stated in the introduction will follow from the choice $\varphi:= +\infty 1_{(\delta,+\infty)}$, in which case $\varphi^\ast(y)=\delta y$ for all $y\geq 0$.
Notice that the supremum in \eqref{def:S} can also be taken over the set
\begin{align}
\label{eq:def.delta.t}
\Delta_{f,t} := \{\nu \in \mathcal{P}_p(\mathbb{R}^d) : \varphi_t(\mathcal{W}_p(\mu_t, \nu)) \leq \|f\|_\infty + 1\}.
\end{align}
As $\varphi_t$ is increasing and unbounded, this implies that there is a uniform upper bound on $\mathcal{W}_p(\mu_t, \nu)$ over $\nu\in\Delta_{f,t}$.
Hence, by the following simple observation, the set $\Delta_{f,t}$ is tight, and therefore the supremum in \eqref{def:S} is attained\footnote{Indeed, the set $\Delta_{f,t}$ is weakly compact by Prokhorov's theorem and lower semicontinuity of $\nu\mapsto \varphi_t(\mathcal{W}_p(\mu_t,\nu))$.}.
\begin{lemma}
	\label{lem:markov}
	For every  $\nu\in\mathcal{P}_p(\mathbb{R}^d)$, $t\geq 0$, and $c>0$, we have that
	\[\nu(\{y\in\mathbb{R}^d : |y|\geq c\})
	\leq \frac{1}{c}\Big( \mathcal{W}_p(\mu_t,\nu) + \Big(\int_{\mathbb{R}^d} |x|^p\,\mu_t(dx)\Big)^{1/p} \Big). \]
\end{lemma}
\begin{proof}
	An application of Markov's and H\"older's inequality implies that $\nu(\{y\in\mathbb{R}^d : |y|\geq c\})\leq c^{-1} (\int_{\mathbb{R}^d} |y|^p\,\nu(dy))^{1/p}$.
	The latter equals $c^{-1}\mathcal{W}_p(\delta_0,\nu)$, so that the proof is completed by the triangle inequality for $\mathcal{W}_p$.
\end{proof}

For further reference, we provide the proof of the following simple observation.

\begin{lemma}
	\label{lem:helpphi}
	For every $c\geq 0$, we have that $\lim_{t\downarrow 0}\sup\{ r\geq 0 : \varphi_t(r)\leq c\}=0$.
\end{lemma}
\begin{proof}
	By assumption $x\mapsto \varphi(\max\{x,0\}^{1/p})$ is a convex lower semicontinuous function, which is not constant equal to zero.
	Therefore, by the Fenchel-Moreau theorem there exist $a>0$ and $b\in\mathbb{R}$, such that $\varphi(x^{1/p})\geq a x + b$ for all $x\geq 0$.
	Thus, for every given $r>0$, we conclude that
	\[ \varphi_t(r)
	=t\varphi\big(\tfrac{r}{t}\big)
	\geq ta\big(\tfrac{r}{t} \big)^p + tb. \]
	As $p>1$, this term converges to infinity when $t$ converges to zero.
\end{proof}

Directly from the definition, the operator $S(t)$ has the following properties.
\begin{lemma}
	\label{lem:S.maps.BUC.to.BUC}
	Let $t\geq 0$ and $f,g\in C_0(\mathbb{R}^d)$ such that $f\leq g$.
	Then, $S(t)$ is a convex contraction on $C_0(\mathbb{R}^d)$, which satisfies $S(t) 0 = 0$, $S(t)f\leq S(t)g$. Further, $S(t)f$ has the same modulus of continuity as $f$.
\end{lemma}
\begin{proof}
	It is clear by definition that $S(t)$ is convex and monotone.
	Moreover, as $\inf \varphi_t=0$, it follows that $S(t)0=0$.
	To show that $S(t)$ is a contraction, note that
	\[\int_{\mathbb{R}^d} f(x+y)\,\nu(dy)\leq \int_{\mathbb{R}^d} g(x+y)\,\nu(dy)+ \|f-g\|_\infty\] for all $\nu\in\mathcal{P}_p(\mathbb{R}^d)$ and $x\in\mathbb{R}^d$.
	Hence, $S(t)f(x)\leq S(t)g(x)+ \|f-g\|_\infty$ for all $x\in\mathbb{R}^d$, and changing the role of $f$ and $g$ yields contractivity.
	
	It remains to prove that $S(t)f\in C_0(\mathbb{R}^d)$. First, since $f$ is 
	in particular uniformly continuous it follows that $S(t)f$ is also 
	uniformly continuous. To that end,
	let $\varepsilon>0$ be arbitrary and fix $\delta>0$ such that $|f(x)-f(y)|\leq \varepsilon$ for $x,y\in\mathbb{R}^d$ with $|x-y|\leq\delta$.
	Then, for every such pair $x,y$, contractivity of $S(t)$ implies that
	\begin{align*}
	S(t)f(x)
	&=S(t) f(x+\cdot)(0) \\
	&\leq S(t)f(y+\cdot)(0) + \|f(x+\cdot) - f(y+\cdot)\|_\infty\\
	&\leq S(t)f(y)+\varepsilon.
	\end{align*}
	Replacing the role of $x$ and $y$ shows that $S(t)f$ is uniformly continuous with the same modulus of continuity as $f$.
	
	Second, we prove that $S(t)f$ is vanishing at infinity. Let $\varepsilon>0$ be arbitrary and fix $a\geq 0$ such that $|f(x)|\leq \varepsilon$ for all $x\in\mathbb{R}^d$ with $|x|\geq a$.
	Since $\mathcal{W}_p(\mu_t, \nu)$ is uniformly bounded over $\nu\in\Delta_{f,t}$, it follows form
	Lemma \ref{lem:markov} that there is $b> 0$ such that $\nu(\{ y\in\mathbb{R}^d : |y|> b\})\leq \varepsilon$ uniformly over $\nu\in\Delta_{f,t}$.
	Hence,
	\[
	S(t) f(x) \leq \sup_{\nu \in \Delta_{t}} \int_{\mathbb{R}^d} f(x+y) 1_{\{|y|\leq b\}} + f(x+y) 1_{\{|y|> b\}}  \,\nu(dy)
	\leq \varepsilon + \varepsilon \|f\|_\infty
	\]
	for all $x\in\mathbb{R}^d$ such that $|x|\geq a+b$.
	For the reverse inequality, use that $S(t)f(x)\geq \int_{\mathbb{R}^d} f(x+y)\,\mu_t(dy)$ for all $x\in\mathbb{R}^d$, which follows from $\varphi(0)=0$.
	Therefore the same arguments as above show that $S(t)f(x)\geq -\varepsilon - \varepsilon\|f\|_\infty$ for all $x\in\mathbb{R}^d$ such that $|x|\geq a+b$.
	As $\varepsilon$ was arbitrary, the claim follows.
\end{proof}

At this point we know that $S(t)$ maps $C_0(\mathbb{R}^d)$ to itself, which allows us to define $S(t)\circ S(s)$, or more generally $\mathscr{S}^n$ as in \eqref{eq:def.cal.S.intro}. The following is the key result for our analysis, and allows in particular to define the limit $\lim_{n\to\infty}\mathscr{S}^n$.

\begin{lemma}
	\label{lem:key}
	For every $0 < s < t$ and $f\in C_0(\mathbb{R}^d)$, we have that
	\[ S(s) S(t-s) f \leq S(t) f.\]
	\begin{proof}
		Fix $f\in C_0(\mathbb{R}^d)$ and $x\in\mathbb{R}^d$.
		Let $\nu_s(db)\in\mathcal{P}_p(\mathbb{R}^d)$ such that
		\[ S(s) S(t-s) f(x)=\int_{\mathbb{R}^d} S(t-s) f(x+b)\,\nu_s(db) - \varphi_s(\mathcal{W}_p(\mu_s,\nu_s)) \]
		and let $\gamma_s(da,db)\in\mathcal{P}_p(\mathbb{R}^d\times\mathbb{R}^d)$ be an optimal coupling between $\mu_s(da)$ and $\nu_s(db)$.
		Similarly, for each $b\in\mathbb{R}^d$, let $\nu_{t-s}^b(de)\in\mathcal{P}_p(\mathbb{R}^d)$ be such that
		\begin{align*}
		S(t-s)f(b)&=\int_{\mathbb{R}^d} f(b+e)\,\nu_{t-s}^b(de)- \varphi_{t-s}( \mathcal{W}_p(\mu_{t-s},\nu_{t-s}^b))
		\end{align*}
		and  let $\gamma^b_{t-s}(dc,de)\in\mathcal{P}_p(\mathbb{R}^d\times\mathbb{R}^d)$ be an optimal coupling between $\mu_{t-s}(dc)$ and $\nu_{t-s}^b(de)$. Now define the measure $\gamma_t(dy,dz)\in\mathcal{P}_p(\mathbb{R}^d\times\mathbb{R}^d)$ by
		\[\int_{\mathbb{R}^d\times\mathbb{R}^d} h(y,z)\,\gamma_t(dy,dz):= \int_{\mathbb{R}^d\times\mathbb{R}^d}\int_{\mathbb{R}^d\times\mathbb{R}^d} h(a+c,b+e)\,\gamma_{t-s}^b(dc ,de )\, \gamma_s(da,db)\]
		for all $h\colon\mathbb{R}^d\times\mathbb{R}^d\to\mathbb{R}$ bounded and Borel (we have ignored the fact that $b\mapsto \gamma_{t-s}^b$ needs to be $\gamma_s$-measurable for this expression to make sense, but this can be shown by usual measurable selection arguments).
		Denoting by $\nu_t(dz):=\gamma_t(dz)$ the second marginal of $\gamma_t$, it holds
		\[ S(t) f(x)
		\geq \int_{\mathbb{R}^d} f(x+z)\,\nu_t(dz) - \varphi_t(\mathcal{W}_p(\mu_t,\nu_t)). \]
		Further, $\gamma_t(dy,dz)$ is a coupling between $\mu_t(dy)$ and $\nu_t(dz)$. Indeed,
		by definition $\gamma_t(dz)=\nu_t(dz)$, and $\gamma_t(dy)=\mu_t(dy)$ as
		\begin{align*}
		\gamma_t(A\times\mathbb{R}^d)
		&=\int_{\mathbb{R}^d\times\mathbb{R}^d}\int_{\mathbb{R}^d\times\mathbb{R}^d} 1_{A}(a+c)\,\gamma_{t-s}^b(dc, de)\,\gamma_s(da,db)\\
		&=\int_{\mathbb{R}^d}\int_{\mathbb{R}^d} 1_A(a+c) \,\mu_{t-s}(dc)\,\mu_{s}(da)
		=(\mu_s\ast\mu_{t-s})(A)
		=\mu_t(A)
		\end{align*}
		for every Borel set $A\subset\mathbb{R}^d$. 
		Similarly, we obtain
		\[
		\nu_t(B)=\gamma_t(\mathbb{R}^d\times B)=\int_{\mathbb{R}^d}\int_{\mathbb{R}^d} 1_B(b+e) \,\nu^b_{t-s}(de)\,\nu_{s}(db)
		\]
		for every Borel set $B\subset\mathbb{R}^d$. Moreover, by definition of the $p$-th Wasserstein distance it holds
		\begin{align*}
		&\mathcal{W}_p(\mu_t,\nu_t)
		\leq \Big(\int_{\mathbb{R}^d \times \mathbb{R}^d} |y-z|^p\,\gamma_t(dy,dz)\Big)^{1/p}\\
		&= \Big(\int_{\mathbb{R}^d \times \mathbb{R}^d}\int_{\mathbb{R}^d \times \mathbb{R}^d} |(b-a)+(e-c)|^p\,\gamma_{t-s}^b(dc,de)\gamma_s(da,db)\Big)^{1/p} \\
		&\leq \Big( \int_{\mathbb{R}^d \times \mathbb{R}^d} |b-a|^p\,\gamma_s(da,db)\Big)^{1/p} + \Big(\int_{\mathbb{R}^d \times \mathbb{R}^d}\int_{\mathbb{R}^d \times \mathbb{R}^d} |e-c|^p\,\gamma_{t-s}^b(dc,de)\,\gamma_s(da,db)\Big)^{1/p}.
		\end{align*}
		Denote by $I$ the first term in the above equation and by $J$ the second one.
		By definition of $\varphi_t=t\varphi(\cdot/t)$, together with the fact that $\varphi$ is convex and increasing, it holds
		\begin{align*}
		\varphi_t(\mathcal{W}_p(\mu_t,\nu_t))
		&\leq t\varphi\Big(\frac{s}{t}\frac{1}{s} I + \frac{t-s}{t}\frac{1}{t-s} J\Big) \\
		&\leq s\varphi\Big(\frac{1}{s} I\Big) + (t-s)\varphi\Big(\frac{1}{t-s} J\Big)
		=\varphi_s(I)+\varphi_{t-s}(J).
		\end{align*}
		Moreover, convexity of $x\mapsto \varphi(x^{1/p})$ implies convexity of $x\mapsto \varphi_{t-s}(x^{1/p})$.
		Therefore, by Jensen's inequality, we obtain
		\[ \varphi_{t-s}(J)\leq \int_{\mathbb{R}^d\times\mathbb{R}^d} \varphi_{t-s}\bigg( \Big(\int_{\mathbb{R}^d\times\mathbb{R}^d} |e-c|^p\, \gamma_{t-s}^b(dc,de) \Big)^{1/p}\bigg)\,\gamma_s(da,db).\]
		Recalling the definitions of $I$ and $J$ and that $\gamma_s$ and $\gamma^b_{t-s}$ are optimal couplings, we conclude
		\begin{align*}
		\varphi_t(\mathcal{W}_p(\mu_t,\nu_t))
		&\leq\varphi_s(\mathcal{W}_p(\mu_s,\nu_s)) +\int_{\mathbb{R}^d} \varphi_{t-s}(\mathcal{W}_p(\mu_{t-s},\nu_{t-s}^b))\,\nu_s(db).
		\end{align*}
		Putting everything together, we obtain
		\begin{align*}
		&S(t) f(x)
		\geq \int_{\mathbb{R}^d} f(x+z)\,\nu_t(dz) - \varphi_s(\mathcal{W}_p(\mu_s,\nu_s)) -\int_{\mathbb{R}^d} \varphi_{t-s}(\mathcal{W}_p(\mu_{t-s},\nu_{t-s}^{b+x}))\,\nu_s(db) \\
		&= \int_{\mathbb{R}^d}\Big(\int_{\mathbb{R}^d} f(x+b+e)\,\nu^{b+x}_{t-s}(de) -\varphi_{t-s}(\mathcal{W}_p(\mu_{t-s},\nu_{t-s}^{b+x}))\Big)\,\nu_s(db)  - \varphi_s(\mathcal{W}_p(\mu_s,\nu_s)) \\
		&= S(s) S(t-s) f(x)
		\end{align*}
		This completes the proof.
	\end{proof}
\end{lemma}

\begin{lemma}
	\label{lem:monotonicity}
	For all $f\in C_0(\mathbb{R}^d)$, $t\ge 0$ and $n\in\mathbb{N}$, we have that $\mathscr{S}^{n+1}(t) f \leq \mathscr{S}^{n}(t) f$. 
	Further, $\mathscr{S}^{n}(t)$ is a contraction on $C_0(\mathbb{R}^d)$, and $\mathscr{S}^{n}(t) f$ has the same modulus of continuity as $f$.
\end{lemma}
\begin{proof}
	Both statements follow from Lemma \ref{lem:key} (respectively Lemma \ref{lem:S.maps.BUC.to.BUC}) together with an induction.
\end{proof}

\begin{corollary}\label{cor:cal.S.properties}
	Let $t \geq 0$ and $f, g \in C_0(\mathbb{R}^d)$ such that $f\leq g$.
	Then, the pointwise limit $\mathscr{S}(t) f := \lim_{n\rightarrow \infty} \mathscr{S}^n(t) f$ exists and is in fact uniform.
	Moreover, $\mathscr{S}(t)$ is a convex contraction on $C_0(\mathbb{R}^d)$ such that $\mathscr{S}(t) 0 = 0$ and $\mathscr{S}(t)f\leq \mathscr{S}(t)g$.
\end{corollary}
\begin{proof}
	We will use the semigroup $(S^\mu(t))_{t\ge 0}$ corresponding to our initial L\'evy process, given by $S^\mu(t)f(x) := \int f(x+y) \,\mu_t(dy)$ and often use that it is a Feller semigroup (see, e.g., \cite[Theorem 3.1.9]{applebaum2009levy}).
	
	By Lemma \ref{lem:monotonicity}, the sequence $(\mathscr{S}^n(t) f)_{n\in\mathbb{N}}$ is decreasing, hence the limit  $\mathscr{S}(t) f = \lim_{n\to\infty} \mathscr{S}^n(t) f$ exists pointwise. Also, the limit $\mathscr{S}(t)f$ is vanishing at infinity. Indeed, from the semigroup property of $(S^\mu(t))_{t\ge 0}$, it follows that
	$S^\mu(t)f\le \mathscr{S}^n(t)f\le S(t)f$ for all $n\in\mathbb{N}$, and therefore $S^\mu(t)f\le \mathscr{S}(t)f\le S(t)f$.  Since by Lemma \ref{lem:S.maps.BUC.to.BUC}, $S(t)f$ is vanishing at infinity, and  $(S^\mu(t))_{t\ge 0}$ is a Feller semigroup, we conclude that  $\mathscr{S}(t)f$ is vanishing at infinity.
	Further, by Lemma \ref{lem:monotonicity}, the sequence $\mathscr{S}^n(t) f$ is uniformly equicontinuous on every compact subset of $\mathbb{R}^d$, which by the Arzel\`a-Ascoli theorem and the fact that $\mathscr{S}(t)f$ is vanishing at infinity, implies that $\lim_{n\to\infty}\|\mathscr{S}^n(t) f - \mathscr{S}(t) f\|_\infty=0$.
	
	Finally, by induction over $n\in\mathbb{N}$, it follows from Lemma \ref{lem:S.maps.BUC.to.BUC} that $\mathscr{S}^n(t)$ is a convex contraction on $C_0(\mathbb{R}^d)$, which satisfies $\mathscr{S}^n(t)0=0$ and $\mathscr{S}^n(t)f\le \mathscr{S}^n(t)g$. These properties remain true for the limit $\mathscr{S}(t)$. The proof is complete.
\end{proof}

Denote by $X_t \sim \mu_t$ the initial L\'evy process.
In the following, we shall often use that $t\mapsto \mu_t$ is continuous w.r.t.\ $\mathcal{W}_p$ (at $t=0$). 
To see that this is true, use the assumption $E[|X_1|^p]<\infty$ and \cite[Theorem 25.18]{sato1999levy} to obtain $E[\sup_{t\in[0,1]} |X_t|^p]<\infty$.
As $X$ has c\`adl\`ag paths, dominated convergence implies that $\mathcal{W}_p(\mu_t,\delta_0)^p=\int_{\mathbb{R}^d} |x|^p\,\mu_t(dx)=E[|X_t|^p]\to 0$ as $t\downarrow0$.

The next result states the strong continuity of the family $(\mathscr{S}(t))_{t\ge 0}$ at zero.
\begin{lemma}
	\label{lem:uniform.at.zero}
	For every $f\in C_0(\mathbb{R}^d)$, we have that
	\[
	\lim_{t\downarrow 0}\|\mathscr{S}(t) f - f\|_\infty = 0.
	\]
	
	\begin{proof}
		Let $f\in C_0(\mathbb{R}^d)$ and $\varepsilon > 0$.
		
		We first show an upper bound, namely that there is $t_0 > 0$ such that $\mathscr{S}(t) f \leq f + 2\varepsilon$ for all $t < t_0$.
		As functions in $C_0(\mathbb{R}^d)$ are uniformly continuous, there is $\delta > 0$ such that $|f(x+y) - f(x)| \leq \varepsilon$ for all $x, y \in \mathbb{R}^d$ with $|y| \leq \delta$. Then,
		\begin{align}
		\label{eq:S(t)n.converges.uniform}
		\begin{split}
		\mathscr{S}(t) f(x)
		&\leq
		S(t) f(x)
		= \sup_{\nu \in \Delta_{f,t}} \Big(\int_{\mathbb{R}^d} f(x+y) \nu(dy) - \varphi_t(\mathcal{W}_p(\mu_t, \nu)) \Big)\\
		&\leq \sup_{\nu\in\Delta_{f,t}} \Big( \int_{\mathbb{R}^d} f(x+y)1_{\{|y|\leq\delta\}} + f(x+y)1_{\{|y|>\delta\}}\,\nu(dy)\Big) \\
		&\leq f(x)+\varepsilon  + \|f\|_\infty \sup_{\nu\in\Delta_{f,t}}  \nu(\{ y\in\mathbb{R}^d : |y|>\delta\}).
		\end{split}
		\end{align}
		By Lemma \ref{lem:helpphi}, it holds that \[
		\lim_{t\downarrow 0}\sup_{\nu \in \Delta_{f,t}} \mathcal{W}_p(\mu_t, \nu)
		\leq \lim_{t\downarrow 0} \sup\big\{r\geq 0 : \varphi_t(r) \leq \|f\|_\infty + 1 \big\}
		= 0.
		\]
		As argued before this lemma, we have $\lim_{t\downarrow 0} \int_{\mathbb{R}^d} |y|^p \,\mu_t(dy)= 0$ and thus it follows from Lemma \ref{lem:markov} that
		\[\lim_{t\downarrow 0}\sup_{\nu \in \Delta_{f,t}}\nu(\{ y\in \mathbb{R}^d : |y| > \delta \}) = 0,\]
		which yields the upper bound.
		
		As for the lower bound, similarly as in the proof of Corollary \ref{cor:cal.S.properties}, we make use of the fact that $S^\mu\le \mathscr{S}$. Since $(S^\mu(t))_{t\ge 0}$ is a Feller semigroup, it holds
		$ f - \varepsilon\le \mathscr{S}(t) f $ for all $0 \leq t < t_0$ for a suitable $t_0>0$.
		This completes the proof.
	\end{proof}
\end{lemma}

For later reference, let us point out that the same proof as given for Lemma \ref{lem:uniform.at.zero} yields the following result.

\begin{corollary}
	\label{cor:uniform.at.zero.ddyadic}
	We have that $ \lim_{n\to\infty}\|\mathscr{S}^n(t_n) f - f\|_\infty = 0$
	for all $f\in C_0(\mathbb{R}^d)$ and all sequences $(t_n)_{n\in\mathbb{N}}$ in $[0, \infty)$ with $\lim_{n\to\infty} t_n=0$.
\end{corollary}

\begin{lemma}
	\label{lem:Snfn.to.Sf}
	Let $t,t_n\geq 0$ with $t_n \in \mathbb{T}_n$, $t_n\le t$, and $g,g_n\in C_0(\mathbb{R}^d)$ for all $n\in\mathbb{N}$. If $\lim_{n\to\infty} t_n = t$ and $\lim_{n\to \infty} \|g_n-g\|_\infty=0$, then
	\[ \lim_{n\to\infty}\| \mathscr{S}^n(t_n) g_n- \mathscr{S}(t)g\|_\infty\to 0.\]
\end{lemma}
\begin{proof}
	As $  \mathscr{S}^n(t)= \mathscr{S}^n(t_n)\mathscr{S}^n(t-t_n)$ by definition, the triangle inequality implies that
	\begin{align*}
	&\| \mathscr{S}(t)g - \mathscr{S}^n(t_n)g_n\|_\infty
	\leq \| \mathscr{S}(t)g - \mathscr{S}^n(t)g\|_\infty\\
	&\quad  +   \| \mathscr{S}^n(t_n)\mathscr{S}^n(t-t_n)g - \mathscr{S}^n(t_n)g\|_\infty + \| \mathscr{S}^n(t_n)g - \mathscr{S}^n(t_n)g_n\|_\infty.
	\end{align*}
	By Corollary \ref{cor:cal.S.properties} the first term converges to zero as $n\to\infty$. As for the middle term, by Lemma \ref{lem:monotonicity} we have that $\mathscr{S}^n(t_n)$ is a contraction, so that
	\[ \| \mathscr{S}^n(t_n)\mathscr{S}^n(t-t_n) g - \mathscr{S}^n(t_n)g\|_\infty
	\leq \|\mathscr{S}^n(t-t_n) g - g \|_\infty . \]
	The latter converges to zero by Corollary \ref{cor:uniform.at.zero.ddyadic}. Again by Lemma \ref{lem:monotonicity}, the last term converges to zero as $n\to\infty$.
	This completes the proof.
\end{proof}

Now, we are ready to state our first main result (the convex generalization of Proposition \ref{prop:semigroup.intro}).

\begin{proposition}
	\label{prop:semigroup.convex}
	The family $(\mathscr{S}(t))_{t\geq 0}$ is a strongly continuous, convex, monotone and normalized contraction semigroup on $C_0(\mathbb{R}^d)$, i.e., for every $s,t\ge 0$ and $f\in C_0(\mathbb{R}^d)$, we have that
	\begin{enumerate}[(i)]
		\item $\mathscr{S}(t)\colon C_0(\mathbb{R}^d)\to C_0(\mathbb{R}^d)$ is a convex and monotone contraction such that $\mathscr{S}(t)0=0$,
		\item $\mathscr{S}(0)f=f$,
		\item $\mathscr{S}(t)\circ\mathscr{S}(s)=\mathscr{S}(t+s)$,
		\item $\lim_{t \downarrow 0} \|\mathscr{S}(t)f-f\|_\infty=0$.
	\end{enumerate}
\end{proposition}
\begin{proof}
	In view of Corollary \ref{cor:cal.S.properties} and Lemma \ref{lem:uniform.at.zero}, it remains to prove the semigroup property $\mathscr{S}(t) \circ \mathscr{S}(s) = \mathscr{S}(t+s)$.
	To that end, fix some $f\in C_0(\mathbb{R}^d)$ and $s,t\geq 0$ and set denote by $s_n,t_n\in \mathbb{T}_n$ the closest dyadic elements prior to $s$ and $t$, respectively.
	By Lemma \ref{lem:Snfn.to.Sf} (applied with $g=g_n=f)$ we have
	\[ \mathscr{S}(t+s)f
	=\lim_{n\to\infty} \mathscr{S}^n(t_n+s_n)f=\lim_{n\to\infty} \mathscr{S}^n(t_n)\circ \mathscr{S}^n(s_n)f, \]
	where the last equality follows by definition of $\mathscr{S}^n$. Further, Lemma \ref{lem:Snfn.to.Sf} also implies that $\mathscr{S}^n(s_n)f$ converges uniformly to $\mathscr{S}(s)f$.
	Therefore, we may apply Lemma \ref{lem:Snfn.to.Sf} again (with 
	$g=\mathscr{S}(s) f$, and $g_n=\mathscr{S}^n(s_n)f$) and obtain
	\[
	\lim_{n\to\infty} \mathscr{S}^n(t_n)\circ\mathscr{S}^n(s_n)f= \mathscr{S}(t)\circ\mathscr{S}(s)f.\]
	This completes the proof.
\end{proof}

\begin{proposition}\label{prop:generator}
	For every $f\in D(A^\mu)\cap C_0^1(\mathbb{R}^d)$, we have that
	\[ \mathscr{A}f
	:=\lim_{t\downarrow 0} \frac{\mathscr{S}(t) f-f}{t}
	= A^\mu f + \varphi^\ast(|\nabla f|)\]
	and the limit is uniform.
\end{proposition}
\begin{proof}
	Fix $f\in D(A^\mu)\cap C_0^1(\mathbb{R}^d)$.
	
	(a)
	We start by showing that
	\begin{align}
	\label{eq:generator.geq}
	\mathscr{S}(t) f- f
	\geq tA^\mu f +  t\varphi^\ast(|\nabla f|) + o(t) \quad\mbox{as }t\downarrow 0.
	\end{align}
	To that end, let $t> 0$.
	For notational simplicity we assume that $t$ is a dyadic number, say $t=k_02^{-n_0}$ for some $k_0,n_0\in\mathbb{N}$; the general case (is only notationally heavier but) works analogously. 
	Then, $\mathscr{S}^{n_0}(t)$ is just the convolution of $S(2^{-n_0})$ with itself $k_0$ times. 
	For every $x\in\mathbb{R}^d$, let $r=r_x\in\mathbb{R}^d$ with $|r|=1$ and $a=a_x\geq 0$ be such that
	\begin{align}
	\label{eq:def.ax.r}
	r\nabla f(x)=|\nabla f(x)| \quad\text{and}\quad \varphi^\ast(|\nabla f(x)|)=a |\nabla f(x)|-\varphi(a),
	\end{align}
	where the product between elements in $\mathbb{R}^d$ is understood as the scalar product. 
	Note that such $r$ exists as $|\cdot|$ is its own dual norm, and such $a$ exists as $\lim_{y\to\infty }\varphi(y)/y=\infty$ which follows from the assumption that $y\mapsto \varphi(y^{1/p})$ is convex. Moreover, since $|\nabla f(x)|$ is uniformly bounded over $x\in\mathbb{R}^d$, the same holds for $a=a_x$.
	
	Now, for each $n\geq n_0$, set
	\[\nu^n_{2^{-n}}:=\mu_{2^{-n}}\ast \delta_{a2^{-n}r}.\]
	Then, one can compute that $\mathcal{W}_p(\mu_{2^{-n}}, \nu_{2^{-n}})=a2^{-n}$, and therefore
	\begin{align}
	\label{eq:estimate.S.below.generator}
	S(2^{-n}) g(y)
	\geq \int_{\mathbb{R}^d} g(y+z+ a2^{-n} r)\,\mu_{2^{-n}}(dz) -\varphi_{2^{-n}}(a2^{-n})
	\end{align}
	for all $g\in C_0(\mathbb{R}^d)$ and $y\in\mathbb{R}^d$. Note that $t= k_n 2^{-n}$ for $k_n:= k_02^{n-n_0}$, and that the measure which results in taking the convolution of $\nu_{2^{-n}}$ with itself $k_n$ times, is equal to $\mu_t\ast \delta_{at r}$.
	As further,
	\[k_n\varphi_{2^{-n}}(a 2^{-n})
	=k_0 2^{n-n_0} 2^{-n}\varphi(a )
	=t\varphi(a)
	=\varphi_t(at)\]
	and each $\varphi_{2^{-n}}(a 2^{-n})$ does not depend on the state variable, estimating every $S(2^{-n})$ which appears in the definition of $\mathscr{S}^n(t)$ (as the convolution of $S(2^{-n})$ with itself $k_n$ times) by \eqref{eq:estimate.S.below.generator} gives
	\begin{align*}
	\mathscr{S}^n(t) f(x)
	&\geq \int_{\mathbb{R}^d} f(x+y+atr )\,\mu_t(dy) - \varphi_t(at)
	\end{align*}
	for all $n\geq n_0$.
	The right hand side does not depend on $n$, so that the definition of $\mathscr{S}(t)f$ as the limit of $\mathscr{S}^n(t) f$ therefore implies that
	\begin{align*}
	\mathscr{S}(t)f(x)-f(x)
	&\geq \int_{\mathbb{R}^d} f(x+y)-f(x)\,\mu_t(dy) \\
	&\qquad +\int_{\mathbb{R}^d}  f( x + y + atr ) - f(x+y)\,\mu_t(dy) - \varphi_t(at)
	=:I_1+I_2.
	\end{align*}
	By definition of the infinitesimal generator $A^\mu$ of $(S^\mu(t))_{t\ge 0}$ and $f\in D(A^\mu)$, the first term $I_1$ equals $t A^\mu f+o(t)$ (uniformly over $x\in\mathbb{R}^d$).
	The second term $I_2$ is estimated by a Taylor's expansion:
	for some (measurable) $\xi=\xi(x,y)$ with $|\xi|\leq ta$, we may write
	\begin{align*}
	I_2&=\int_{\mathbb{R}^d} atr \nabla f(x+y +\xi(x,y)) \,\mu_t(dy) - t\varphi(a)\\
	&\geq ar \nabla f(x) + o(1)-t\varphi(a),
	\end{align*}
	uniformly over $x\in\mathbb{R}^d$, where we need to justify the last step.
	Indeed, this follows as in the proof of Lemma \ref{lem:uniform.at.zero} by splitting the $\mu_t(dy)$ integral into two parts (close to zero $\{|y|\le b\}$ and its complement $\{|y|> b\}$), and using  uniform continuity of $ar \nabla f(x+\cdot)$    together with the fact that $\mu_t(\{|y|>b\})\to 0$ and $\lim_{t\downarrow 0}\sup_{x,y\in\mathbb{R}^d}|\xi(x,y)|=0$ as $a=a_x$ is bounded uniformly over $x\in\mathbb{R}^d$.
	Recalling \eqref{eq:def.ax.r} we conclude that
	\[\int_{\mathbb{R}^d}  f(x+y+ atr ) - f(x+y)\,\mu_t(dy) - \varphi_t(at)
	\ge t\varphi^\ast(|\nabla f(x)|)+o(t)\]
	uniformly over $x\in\mathbb{R}^d$, which shows \eqref{eq:generator.geq}.
	
	(b) It remains to show that
	\begin{align}
	\label{eq:generator.leq}
	\mathscr{S}(t) f- f
	\leq t A^\mu f + t\varphi^\ast(|\nabla f|)+o(t).
	\end{align}
	Since $\int_{\mathbb{R}^d} f(x+y)\,\mu_t(dy)-f(x)=tA^\mu f(x) + o(t)$ as $f\in D(A^\mu)$ and $\mathscr{S}(t)f\leq S(t) f$ by Corollary \ref{cor:cal.S.properties}, it holds
	\begin{align*}
	&\mathscr{S}(t) f(x) - f(x)
	\leq S(t)f (x) - \int_{\mathbb{R}^d} f(x+y)\,\mu_t(dy) + tAf(x)+o(t) \\
	&=\sup_{u,\nu}	\Big(\int_{\mathbb{R}^d} f(x+z)\,\nu(dz)-\int_{\mathbb{R}^d} f(x+y)\,\mu_t(dy)  - t\varphi\big(\tfrac{u}{t}\big)\Big) + tAf(x)+o(t)
	\end{align*}
	uniformly over $x\in\mathbb{R}^d$, where the supremum is taken over all $u\geq 0$ and $\nu \in\mathcal{P}_p(\mathbb{R}^d)$ with $\mathcal{W}_p(\mu_t,\nu)=u$.
	Actually, for every $t\geq 0$,  one may restrict to those $u\geq 0$ for which $t\varphi(u/t)\leq \|f\|_\infty +1$.
	As $\varphi$ grows faster than linear, this implies that there is some $u_0$ (independent of $t$) for which the latter implies $u\leq u_0t$.
	
	Now, fix $0\leq u\leq u_0t$ and $\nu\in\mathcal{P}_p(\mathbb{R}^d)$ with $\mathcal{W}_p(\mu_t,\nu)=u$, and a coupling $\pi(dy,dz)$ between $\mu_t$ and $\nu$ which is optimal for $\mathcal{W}_p(\mu_t,\nu)$.
	By Taylor's theorem,
	\[
	f(x+z)-f(x+y)=\nabla f(x+y+\xi)(z-y)
	\]
	for all $x,y,z\in\mathbb{R}^d$, where  $\xi=\xi(x,y,z)$
	is a measurable function such that $|\xi|\le|z-y|$. Hence, it follows from H\"older's inequality that
	\begin{align*}
	&\int_{\mathbb{R}^d} f(x+z)\,\nu(dz)-\int_{\mathbb{R}^d} f(x+y)\,\mu_t(dy)
	=\int_{\mathbb{R}^d} \nabla f(x+\xi)(z-y)\,\pi(dy,dz)\\
	\leq& \Big(\int_{\mathbb{R}^d\times \mathbb{R}^d} |\nabla f(x+\xi)|^{p^\ast}\,\pi(dy,dz)\Big)^{1/p^\ast} \Big(\int_{\mathbb{R}^d\times \mathbb{R}^d} |z-y|^{p}\,\pi(dy,dz)\Big)^{1/p},
	\end{align*}
	where $p^\ast=p/(p-1)$ is the conjugate H\"older exponent of $p$.
	For every $0\leq u\leq u_0t$ and $\nu$ as above, it follows from  $\mathcal{W}_p(\delta_0,\nu)\leq \mathcal{W}_p(\delta_0,\mu_t) + u_0t=o(1)$, that
	\[ \Big(\int_{\mathbb{R}^d\times \mathbb{R}^d} |\nabla f(x+\xi)|^{p^\ast}\,\pi(dy,dz)\Big)^{1/p^\ast}
	\leq |\nabla f(x)|+ o(1) \]
	uniformly over $x\in\mathbb{R}^d$, again by the same arguments as in the proof of Lemma \ref{lem:uniform.at.zero}.
	Putting everything together yields
	\begin{align*}
	& \frac{1}{t}\Big( S(t) f(x) - \int_{\mathbb{R}^d} f(x+y)\,\mu_t(dy)\Big) \\
	&\leq \sup_{0\leq u\leq u_0 t} \Big(\frac{u}{t} (|\nabla f(x)| + o(1))  - \varphi\Big(\frac{u}{t}\Big)\Big)\leq \varphi^\ast( |\nabla f(x)| )
	\end{align*}
	uniformly over $x\in\mathbb{R}^d$, where the last inequality follows from the definition of the convex conjugate $\varphi^\ast$.
	This shows \eqref{eq:generator.leq} and therefore completes the proof.
\end{proof}

This is a good place to mention the recent paper \cite{bartl2020robust} in which related ideas as in Proposition \ref{prop:generator} are applied in the context of stochastic optimization.
With Proposition \ref{prop:generator} at our disposal, we can finally prove our main result (Theorem \ref{thm:pde.intro}), or rather its convex generalization (Theorem \ref{thm:pde.convex} below).

Before doing so, let us recall the notion of viscosity solution that we use:
denote by $C^1_0(\mathbb{R}^d)$ the space of all continuously differentiable 
functions $f\in C_0(\mathbb{R}^d)$ whose gradient is vanishing at infinity, and 
call $v\colon (0,\infty)\to C_0(\mathbb{R}^d)$ \emph{test function} if it is 
differentiable (w.r.t.\ the supremum norm) and satisfies $v(t)\in D(A^\mu)\cap 
C_0^1(\mathbb{R}^d)$ for every $t\in(0,\infty)$.

Then, following \cite{denk2020semigroup}, we say that a continuous function $u: [0, \infty) \rightarrow C_0(\mathbb{R}^d)$ is a \emph{viscosity subsolution} of 
\[	\begin{cases}
\partial_t u(t,x) = \mathscr{A} u(t,x) &\text{for } (t,x)\in (0,\infty)\times\mathbb{R}^d,\\
u(0,x)=f(x)&\text{for }x\in\mathbb{R}^d,
\end{cases}\]
if for every $(t,x) \in (0,\infty)\times\mathbb{R}^d$ and every test function $v$ satisfying $u\leq v$ and $v(t,x)=u(t,x)$, it holds that $\partial_t v(t,x)\le \mathscr{A}v(t,x)$. Similarly, $u$ is called \emph{viscosity supersolution} if the above holds with `$\leq$' replaced by '$\geq$' at both instances, and a \emph{viscosity solution} if it is both a viscosity supersolution  and subsolution.

As a consequence of the previous result we derive the following: 

\begin{theorem}
	\label{thm:pde.convex}
	Let $f\in C_0(\mathbb{R}^d)$ and define $u\colon[0,\infty)\times\mathbb{R}^d\to\mathbb{R}$  via $u(t,x):=\mathscr{S}(t)f(x)$.
	Then $u$ is a viscosity solution of
	\begin{align}
	\label{eq:PDE}
	\begin{cases}
	\partial_t u(t,x) = A^\mu u(t,x)  + \varphi^\ast(|\nabla u(t,x)|) &\text{for } (t,x)\in (0,\infty)\times\mathbb{R}^d,\\
	u(0,x)=f(x)&\text{for }x\in\mathbb{R}^d.
	\end{cases}
	\end{align}
\end{theorem}
\begin{proof}
	To show that $u$ is a viscosity subsolution, let $v$ be a test function such that $u\le v$ and $v(t,x)=u(t,x)$ for some $(t,x)\in(0,\infty)\times\mathbb{R}^d$.
	Since $v$ is differentiable at $t$ there exists $\partial_tv(t)\in C_0(\mathbb{R}^d)$ such that $v(t+h)=v(t)+h\partial_t v(t) +o(h)$ for $|h|\to 0$.
	Similar to the proof of Lemma \ref{lem:uniform.at.zero} it follows that
	\[
	\mathscr{S}(h)\big(v(t)-h\partial_t v(t)+o(h)\big)-S(h)(v(t))=-h\partial_t v(t)+o(h)
	\]
	for $|h|\to 0$. 
	Hence, for $h>0$ small enough, using Proposition \ref{prop:semigroup.convex}, we have that
	\begin{align*}
	0&=\frac{\mathscr{S}(h)\mathscr{S}(t-h)f-\mathscr{S}(t)f}{h}
	=\frac{\mathscr{S}(h)u(t-h)-u(t)}{h}
	\le \frac{\mathscr{S}(h)v(t-h)-u(t)}{h}\\
	&= \frac{\mathscr{S}(h)\big(v(t)-h\partial_tv(t)+o(h)\big)-\mathscr{S}(h)v(t)}{h}+	\frac{\mathscr{S}(h)v(t)-u(t)}{h}\\
	&=-\partial_tv(t)+ \frac{\mathscr{S}(h)v(t)-v(t)}{h} + \frac{v(t)-u(t)}{h} + o(h).
	\end{align*}
	In particular, since  $h^{-1}( \mathscr{S}(h)v(t,x)-v(t,x))\to \mathscr{A}v(t,x)$ uniformly over $x\in\mathbb{R}^d$ by Proposition \ref{prop:generator}, and $v(t,x)=u(t,x)$ by assumption, we conclude that $0\le -\partial_tv(t,x) + \mathscr{A}v(t,x)$.
	This shows that $u$ is a viscosity subsolution of  \eqref{eq:PDE}.
	The arguments that $u$ is a viscosity supersolution follows along the same lines  successively replacing `$\leq$' by `$\geq$'.
\end{proof}

Finally, we sketch how uniqueness of the solution of \eqref{eq:PDE} in 	Theorem \ref{thm:pde.convex} may be obtained.
Under certain conditions on the initial L\'evy process, one obtains from \cite[Corollary 53]{hu2009gl}
uniqueness of the viscosity solution \eqref{eq:PDE} by using the space $C_b^{2, 3}((0, \infty)\times \mathbb{R}^d)$ as test functions. This requires an extension of the semigroup $(\mathscr{S}(t))_{t\ge 0}$ to the space $\BUC(\mathbb{R}^d)$ of all bounded and uniformly continuous functions, which may be achieved via monotone approximation and continuity from above of the operators $\mathscr{S}(t)$, $t\ge 0$, see also \cite[Remark 5.4]{denk2018kolmogorov} Then, by adapting Proposition \ref{prop:generator}, it follows that Theorem \ref{thm:pde.convex} also holds for the test functions $v \colon (0, \infty) \rightarrow \BUC(\mathbb{R}^d)$ which are differentiable and $v(t)\in \BUC^2(\mathbb{R}^d)$ for all $t\ge 0$, where $\BUC^2(\mathbb{R}^d)$ denotes the space of all functions which are twice differentiable with bounded uniformly continuous derivatives up to order 2. Once this is done, the results in \cite{hu2009gl} may be used, since $C_b^{2, 3}((0, \infty)\times \mathbb{R}^d)$ is a subset of the considered test functions, see \cite[Remark 2.7]{denk2020semigroup}.



\bibliographystyle{abbrv}




\vspace{1em}
\noindent
\textsc{Acknowledgements:}
%
We thank the editor and an anonymous referee of \textsl{ECP} for valuable comments and feedback on an earlier version of the paper. Daniel Bartl is grateful for financial support through the Vienna Science and Technology Fund (WWTF) project MA16-021 and the Austrian Science Fund (FWF) project P28661.

\end{document}